\documentclass[12pt,twoside]{article}
\usepackage{amsmath}
\usepackage{graphicx}
\usepackage{yhmath}
\usepackage{mathdots}
\usepackage[nottoc,numbib]{tocbibind}
\usepackage{mathtools}
\usepackage[all,cmtip]{xy}
 \usepackage{amsfonts}
\usepackage{amsthm}
\usepackage{amsmath,amscd}

\usepackage{amssymb}
\usepackage{tikz}
\usetikzlibrary{arrows.meta, decorations.markings}
\date{June 18, 2021}

\begin{document}

\centerline {\Large{\bf Rooted mutation groups and finite type   cluster algebras}}

\centerline{}

\centerline{}

\centerline{\bf {Ibrahim Saleh}}

\centerline{Email: salehi@uww.edu}
\centerline{University of Wisconsin Whitewater}

 \newtheorem{thm}{Theorem}[section]
 \newtheorem{cor}[thm]{Corollary}
  \newtheorem{cor and defn}[thm]{Corollary and Definition}
 \newtheorem{lem}[thm]{Lemma}
 \newtheorem{prop}[thm]{Proposition}
 \theoremstyle{definition}
 \newtheorem{defn}[thm]{Definition}
 \newtheorem{defns}[thm]{Definitions}
 \newtheorem{defns and nots}[thm]{Definitions and Notations}
 \theoremstyle{remark}
 \newtheorem{rem}[thm]{Remark}
 \newtheorem{rems}[thm]{Remarks}
 \newtheorem{exam}[thm]{Example}
 \newtheorem{exams}[thm]{Examples}
 \newtheorem{conj}[thm]{Conjecture}
 \newtheorem{que}[thm]{Question}
  \newtheorem{ques}[thm]{Questions}
  \newtheorem{ques and Conj}[thm]{Questions and Conjecture}
 \newtheorem{rem and def}[thm]{Remark and Definition}
 \newtheorem{def and rem}[thm]{Definition and Remark}
 \newtheorem{corr}[thm]{Corollary of the Proof of Proposition 5.15}
 \numberwithin{equation}{section}
\newtheorem{IbrI}{Lemma}[section]
\newtheorem{chiral}[IbrI]{Definition}
\newtheorem{IbrII}[IbrI]{Lemma}
\newcommand{\field}[1]{\mathbb{#1}}
 \newtheorem{mapProof}[thm]{Map of the Proof}
 \newcommand{\eps}{\varepsilon}
 \newcommand{\To}{\longrightarrow}
 \newcommand{\h}{\mathcal{H}}
 \newcommand{\s}{\mathcal{S}}
 \newcommand{\A}{\mathcal{A}}
 \newcommand{\J}{\mathcal{J}}
 \newcommand{\M}{\mathcal{M}}
 \newcommand{\W}{\mathcal{W}}
 \newcommand{\X}{\mathcal{X}}
 \newcommand{\BOP}{\mathbf{B}}
 \newcommand{\BH}{\mathbf{B}(\mathcal{H})}
 \newcommand{\KH}{\mathcal{K}(\mathcal{H})}
 \newcommand{\Real}{\mathbb{R}}
 \newcommand{\Complex}{\mathbb{C}}
 \newcommand{\Field}{\mathbb{F}}
 \newcommand{\RPlus}{\Real^{+}}
 \newcommand{\Polar}{\mathcal{P}_{\s}}
 \newcommand{\Poly}{\mathcal{P}(E)}
 \newcommand{\EssD}{\mathcal{D}}
 \newcommand{\Lom}{\mathcal{L}}
 \newcommand{\States}{\mathcal{T}}
 \newcommand{\abs}[1]{\left\vert#1\right\vert}
 \newcommand{\set}[1]{\left\{#1\right\}}
 \newcommand{\seq}[1]{\left<#1\right>}
 \newcommand{\norm}[1]{\left\Vert#1\right\Vert}
 \newcommand{\essnorm}[1]{\norm{#1}_{\ess}}
 \newcommand{\beq}{\begin{equation}}
\newcommand{\eeq}{\end{equation}}
\newcommand{\rarr}{\rightarrow}
\newcommand{\cA}{\mathcal{A}}
\newcommand{\cS}{\mathcal{S}}
\newcommand{\cC}{\mathcal{C}}
\newcommand{\cU}{\mathcal{U}}
\newcommand{\cR}{\mathcal{R}}
\newcommand{\RMod}{R\text{-Mod}}
\newcommand{\AMod}{A\text{-Mod}}
\newcommand{\Rep}{\text{Rep}}
\newcommand{\Aut}{\text{Aut}}
\newcommand{\XAut}{\xi\Aut}
\newcommand{\Rtzn}{R \{\theta, z, n \}}
\newcommand{\TxC}{(\Theta, \xi, \cC)}
\newcommand{\Chir}{\text{Chir}}

\tableofcontents \begin{abstract} For a fixed seed $(X, Q)$, a \emph{rooted mutation loop} is a sequence of mutations that preserves $(X, Q)$. The group generated by all rooted mutation loops is called \emph{rooted mutation group} and will be denoted by $\mathcal{M}(Q)$. The \emph{global mutation group} of $(X, Q)$, denoted $\mathcal{M}$, is the group of all mutation sequences subject to the relations on the cluster structure of $(X, Q)$. In this article, we show that two finite type cluster algebras  $\mathcal{A}(Q)$ and  $\mathcal{A}(Q')$ are isomorphic if and only if their rooted mutation groups are isomorphic and the sets $\mathcal{M}/\mathcal{M}(Q)$ and  $\mathcal{M'}/\mathcal{M}(Q')$ are in one to one correspondence. The second main result shows that the group $\mathcal{M}(Q)$ and the set $\mathcal{M}/\mathcal{M}(Q)$ determine the finiteness of the cluster algebra $\mathcal{A}(Q)$ and vice versa.
\end{abstract}



\section{Introduction}

S. Fomin and A. Zelevinsky introduced cluster algebras in [17, 8, 9, 10, 2]. A cluster algebra is a  ring with distinguished sets of generators called ``\emph{clusters}". A cluster is a set of commutative variables that form a transcendental basis for some rational field $\mathcal{F}$. Each cluster is paired with a ``valued" quiver (skew-symmetrizable matrix) to form what is called a \emph{seed}. The \emph{cluster structure} of a cluster algebra is the set of all seeds generated from an initial seed by applying \emph{mutations}. A cluster algebra is of finite type if its cluster structure is a finite set and equivalently the number of clusters is finite.  

For a fixed seed $(X, Q)$, a \emph{rooted mutation loop} is a sequence of mutations that preserves $(X, Q)$. The group generated by all rooted mutation loops is called \emph{rooted mutation group} and will be denoted by $\mathcal{M}(Q)$.  In [16], the concept of global mutation loops was introduced as sequences of mutations preserving every seed in the cluster structure. The group generated by all sequences of mutations subject to the global mutation loops, as relations, is termed the \emph{global mutation group} and denoted as $\mathcal{M}$. 
  
In this article, we used a ``natural" orientation for the cluster pattern  to  introduce a concept termed \emph{reduction process} applied to the directed cluster pattern (refer to Definition 3.7). This procedure forms the basis for the statement of the main result, detailed in Theorem 3.15.  
In [14], an equivalent condition for the isomorphism between two cluster algebras, $\mathcal{A}(Q)$ and $\mathcal{A}(Q')$, was established. This condition implies that the quivers $Q$ and $Q'$ are symmetric, meaning there exists a quiver automorphism $\sigma$ such that $Q' = \sigma(Q)$. We provide another equivalent condition for two finite type cluster algebras to be isomorphic, as detailed in Theorem 3.15. The following presents a statement of this result.

\begin{thm} Two cluster  algebras $\mathcal{A}(X, Q)$ and  $\mathcal{A}(X', Q')$, of finite type, are isomorphic as cluster algebras if and only if the following conditions are satisfied 
\begin{itemize}
  \item the associated rooted mutations groups $\mathcal{M}(Q)$ and $\mathcal{M}(Q')$ are isomorphic;
  \item the two sets $\mathcal{M}/\mathcal{M}(Q)$ and $\mathcal{M'}/\mathcal{M}(Q')$ are in one-to-one correspondence. 
\end{itemize}
\end{thm}

In [9], Theorem 1.4 provides a complete classification of the cluster algebras of finite type. This classification is identical to the Cartan-Killing classification of semisimple Lie algebras and finite root systems. The other main result of this article, namely Theorem 3.12, provides an equivalent condition for a cluster algebra to be finite type. Here is a brief version of the theorem.  

\begin{thm}
The following are equivalent
\begin{enumerate}
  \item The cluster algebra $A(Q)$ is of finite type;
   \item $\mathcal{M}(Q)$ is a finite group and $\mathcal{M}/\mathcal{M}(Q)$ is a finite set.
\end{enumerate}

\end{thm}

The paper is structured as follows: we provide a concise introduction to valued quivers and their mutation. The third section offers an overview of rooted mutation loops and the group generated by them. We introduce a reduction process on the cluster pattern (Definition 3.7) and employ the resulting cluster diagram to prove the main results, including Theorems 3.12 and 3.15. We will use $\mathcal{F}$ as a rational field in $n$ indeterminant over a filed $K$ of  zero characteristic.

\section{Valued Quivers Mutation}

\begin{defns}
\begin{enumerate}
 
  \item \emph{An oriented valued quiver} of rank $n$ is a quadruple $Q=(Q_{0}, Q_{1}, V, d)$, where
  \begin{itemize}
 \item $Q_{0}$ is a  set of $n$ vertices labeled by $[1, n]$.
  \item $Q_{1}$ is a set of ordered pairs of vertices, that is $Q_{1}\subset Q_{0}\times Q_{0}$ such that; $(i, i)\notin Q_{1}$ for every $i\in Q_{0}$, and if $(i, j)\in Q_{1}$, then $(j, i)\notin Q_{1}$.
  \item $V=\{(d_{ij}, d_{ji})\in \mathbb{N}\times\mathbb{N} | (i, j)\in Q_{1}\}$, $V$ is called the valuation of $Q$. The weight of an edge $\alpha=(i, j)$ is the product $d_{ij}d_{ji}$ and is denoted by  $w_{i, j}$. \emph{The weight} of $Q$ is given by $w(Q)=max\{w_{ij}; (i, j) \in Q_{1}\}$.
  \item $d=(d_{1}, \cdots, d_{n})$ where $d_{i}$ is a positive integer for each $i$, such that $d_{i}d_{ij}=d_{ji}d_{j}$ for every $i, j\in [1, n]$.
   \end{itemize}
  In the case of $(i, j)\in Q_{1}$, then there is an arrow oriented from $i$ to $j$, and in notation, we shall use the symbol $\xymatrix{{\cdot}_{i} \ar[r]^{(d_{ij}, d_{ji})}&{\cdot}_{j}}$, or $\xymatrix{{\cdot}_{i} \ar[r]^{w_{i, j}}&{\cdot}_{j}}$ when the emphasis is on the weight of the edge. We will also use $rk(Q)$ for the rank of $Q$. The quiver $Q$ will be called \emph{simply-laced} if $w(Q)=1$. We will ignore labeling any edge of weight one. Also, a vertex $i\in Q_{0}$ is called a \emph{leaf } if there is exactly one vertex $j$ such that $w_{ij}\neq 0$ and $w_{kj}=0$, for all $k\in[1, n] \backslash \{j\}$.

  The neighborhood of the vertex $i$ in the quiver $Q$, denoted as $Q_{Nhb. _i}$, is defined as the set of all vertices connected to it, i. e. , $Q_{Nhb. i}=\{j\in Q_{0}; w_{ij}\neq 0\}$. 

      \item Let $\mathcal{S}_{n}$ be the symmetric group in $n$ letters. One can introduce an action of  $\mathcal{S}_{n}$ in the set of quivers of rank $n$ as follows: for a permutation $\tau$, the quiver $\tau(Q)$  is  obtained from $Q$ by permuting the vertices of $Q$ using $\tau$ such that for every edge $\xymatrix{{\cdot}_{i} \ar[r]&{\cdot}_{j}}$ in $Q$, the valuation $(d_{ij}, d_{ji})$ is assigned  to the edge $\xymatrix{{\cdot}_{\tau(i)} \ar[r]&{\cdot}_{\tau(j)}}$ in $\tau(Q)$. In such case, we say that $\tau(Q)$  is symmetric to $Q$.

\end{enumerate}
\end{defns}
We note that every oriented valued quiver corresponds to a skew symmetrizable matrix   $B(Q)=(b_{ij})$ given by
\begin{equation}\label{}
  b_{ij}=\begin{cases} d_{ij}, & \text{ if }(i, j)\in Q_{1}, \\
    0, & \text{ if }i=j, \\
-d_{ij}, & \text{ if }(j, i)\in Q_{1}.
    \end{cases}
\end{equation}
One can also see that every skew symmetrizable matrix $B$ corresponds to an oriented valued quiver $Q$ such that $B(Q)=B$.

 All our valued quivers are oriented, so in the rest of the paper, we will omit the word ``oriented". All quivers are of rank $n$ unless stated otherwise.
 We will also remove the word valued from the term ``valued quiver" when there is no confusion.

   \begin{defn}[\emph{Valued quiver mutation}]
 Let $Q$ be a valued quiver. The mutation  $\mu_{k}(Q)$ at a vertex  $k$  is defined through Fomin-Zelevinsky's mutation of the associated skew-symmetrizable matrix. The mutation of a skew symmetrizable matrix $B=(b_{ij})$ on the direction $k\in [1, n]$ is given by $\mu_{k}(B)=(b'_{ij})$, where
\begin{equation}
b'_{ij}=\begin{cases} -b_{ij}, & \text{if} \ k \in \{i, j\}, \\
   b_{ij}+\text{sign}(b_{ik})\max(0, b_{ik}b_{kj}), & \text{otherwise. }
   \end{cases}
  \end{equation}

\end{defn}
The following remarks provide a set of rules that are adequate to calculate mutations of valued quivers without using their associated skew-symmetrizable matrix.
\begin{rems}
\begin{enumerate}

  \item Let $Q=(Q_{0}, Q_{1}, V, d)$ be a valued quiver. The mutation $\mu_{k}(Q)$ at the  vertex  $k$   is described using the mutation of $B(Q)$ as follows: Let $\mu _{k}(Q)=(Q_{0}, Q'_{1}, V', d)$, we obtain $Q'_{1}$ and $V'$,   by altering  $Q_{1}$ and $V$, based on the following rules.
\begin{enumerate}
 \item replace the  pairs $(i, k)$ and $(k, j)$  with $(k, i)$ and $(j, k)$  respectively and, in the same manner, switch the components of the  ordered pairs of their valuations;
  \item if  $(i, k), (k, j)\in Q_{1}$, such that neither of $(j, i)$  or $(i, j)$ is in $Q_{1}$ (respectively $(i, j)\in Q_{1}$) add the  pair $(i, j)$ to $Q'_{1}$, and give it the valuation $(v_{ik}v_{kj}, v_{ki}v_{jk})$ (respectively change its valuation to $(v_{ij}+v_{ik}v_{kj}, v_{ji}+v_{ki}v_{jk})$);
 \item if $(i, k)$, $(k, j)$ and $(j, i)$ in $Q_{1}$, then we have three cases
 \begin{enumerate}
   \item if $v_{ik}v_{kj}<v_{ij}$, then keep $(j, i)$ and change its valuation to $(v_{ji}-v_{jk}v_{ki}, -v_{ij}+v_{ik}v_{kj})$;
   \item if $v_{ik}v_{kj}>v_{ij}$, then replace $(j, i)$ with $(i, j)$ and change its valuation to $(-v_{ij}+v_{ik}v_{kj}, |v_{ji}-v_{jk}v_{ki}|)$;
   \item if $v_{ik}v_{kj}=v_{ij}$,  then remove $(j, i)$ and its valuation.
 \end{enumerate}
 \item $d$ will stay the same in $\mu_{k}(Q)$.
\end{enumerate}

 \item One can see that; $\mu^{2}_{k}(Q)=Q$ and  $\mu_{k}(B(Q))=B(\mu_{k}(Q))$ at each vertex  $k\in [1, n]$ where $\mu_{k}(B(Q))$ is the mutation of the matrix $B(Q)$. For more information on mutations of skew-symmetrizable matrices, see, for example, [10, 14].
 \end{enumerate}
\end{rems}

\begin{defns}
\begin{itemize}
           
      \item \emph{A seed } in $\mathcal{F}$ of rank $n$ is a pair $(X, Q)$, where
   \begin{enumerate}
     \item The $n$-tuple $X=(x_{1}, \ldots, x_{n})$ is called a \emph{ cluster} where  $X=(x_{1}, x_{2}, \ldots, x_{n})\in \mathcal{F}^n$ is a transcendence basis of  $\mathcal{F}$ over $K$  that generates $\mathcal{F}$. Elements of $X$ will be called \emph{cluster variables}.
     \item $Q$ is an oriented valued quiver with $n$ vertices. The vertices of  $Q$ are labeled by numbers from $[1, n]$.

\item If $\mu=\mu_{i_{1}}\ldots \mu_{i_{k}}$ is a sequence of mutations, we will use the notations  $\{\mu\}: =\{\mu_{i_{1}}, \ldots, \mu_{k}\}$ and $\overleftarrow{\mu}: =\mu_{i_{k}}\cdots \mu_{i_{1}}$.
     \end{enumerate}
\end{itemize}
\end{defns}

\begin{defns}[Subquivers][11, 15]
Let $Q$ be a quiver of rank $n$. We can obtain a sub-mutation class of $[Q]$ by fixing a subset $I$ of $Q_{0}$ and applying all possible sequences of mutations at the vertices from the set $I$ only. We will refer to the vertices in the set $Q_{0}\backslash I$ as the \emph{frozen vertices}. In such a case, the quiver $Q$, with a set of frozen variables $Q_{0}\backslash I$, will be denoted by $Q_{I}$ and we will use $Q_{I}\leq Q$ to indicate that $Q_{I}$ is a subquiver of $Q$

\end{defns}

     \begin{exams}
     \begin{enumerate}
       \item Let
     \begin{equation}\label{}
     Q=\xymatrix{ \cdot_{4}& \cdot_{3}   \ar[l]_{(2,3)}\ar[r]^{(2,3)}&   \cdot_{2}\ar[d]^{(1,2)}&\cdot_{7}\ar[l]_{(2,1)}\\
  &\cdot_{6} \ar[r]&\cdot_{1}\ar[ul]^{(6, 2)}\ar[r]^{(2,3)}&\cdot_{5} }
     \end{equation}
     Consider the subquiver $Q_{I}$, with $I=\{1, 2, 3\}$,  and $d=(1, 2, 3)$. So, $rk(Q_{I})=3$ and $w(Q_{I})=12$. Here, $(6,2)$ is the valuation of the edge  $\xymatrix{{\cdot}_{1} \ar[r]&{\cdot}_{3}}$. Applying mutation at the vertex  $2$ produces the following  quiver

  \begin{equation}\label{}
      \nonumber \mu_{2}(Q_{I})=\xymatrix{ \cdot_{4}& \cdot_{3}   \ar[l]_{(2,3)}&   \cdot_{2} \ar[l]_{(3,2)}\ar[r]^{(1,2)}&\cdot_{7}\ar[dl]^{(2,2)}\\
  &\cdot_{6} \ar[r]&\cdot_{1}\ar[u]^{(2,1)}\ar[r]_{(2,3)}&\cdot_{5}. }
     \end{equation}

     \end{enumerate}

  \end{exams}

\begin{lem}  If $Q$ is simply-laced quiver then for every permutation $\sigma$ there is a sequence of mutations $\mu$ such that $\sigma(Q)=\mu(Q)$, for more details see [Theorem 2.6 in [1]].
\end{lem}

\begin{defn}[Seed mutation] Let $(X, Q)$ be a seed in $\mathcal{F}$. For each fixed $k\in [1, n]$,
 we define a new seed $\mu_{k}(X, Q)=(\mu _{k}(X), \mu _{k}(Q))$ by  setting $\mu _{k}(X)=(x'_{1}, \ldots, x'_{n})$  where
\begin{equation}\label{}
   x'_{i}=\begin{cases} x_{i}, & \text{ if } i\neq k, \\
    \frac{ \prod\limits_{b_{ji}> 0} x_{j}^{b_{ji}}+
    \prod\limits_{b_{ji}< 0} x_{j}^{-b_{ji}}}{x_{i}}, & \text{ if }i=k.
    \end{cases}
\end{equation}

And $\mu _{k}(Q)$  is the mutation of $Q$ at the vertex $k\in [1, n]$.

   \end{defn}

\begin{defns}[Cluster structure and cluster algebra]
\begin{itemize}
  \item The set of all seeds obtained by applying all possible sequences of mutations on the seed $(X, Q)$ is called the  \emph{cluster structure} of $(X,Q)$ and it will be denoted by $[(X,Q)]$.

        \item Let $\mathcal{X}$ be the union of all clusters in the cluster structure of $(X,Q)$. The  \emph{rooted cluster algebra} $\mathcal{A}(Q)$ is the $\mathbb{Z}$-subalgebra of $\mathcal{F}$ generated by $\mathcal{X}$. For simplicity we will omit the word ``rooted".

\end{itemize}
\end{defns}
One can see that any seed in the cluster structure of  $(X,Q)$ generates the same cluster structure.

\begin{defn}(\textbf{Rooted cluster digraph of an initial seed $(X,Q)$}). Fix an initial seed $(X,Q)$.
 The \emph{rooted cluster digraph} $\mathbb{G}(Q)$ of the initial seed $(X,Q)$ is an $n-$regular directed ``connected" graph formed as follows
 \begin{itemize}
   \item The vertices are assigned to be the elements of the cluster class $[(X,Q)]$  such that the endpoints of any edge are obtained from each other by the quiver mutation in the direction of the edge label; 
   \item The edges are $2$-cycles, where each $2$-cycles is  associated to one of the single mutations $\mu_{1}, \mu_{2}, \ldots, \mu_{n}$, such that each arrow in the same $2$-cycle is labeled by the same single mutation.  

  So, any two adjacent vertices in $\mathbb{G}(Q)$ would look like the following 
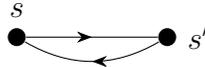
\begin{equation}\label{}
  \nonumber \begin{tikzpicture}[
       decoration = {markings,
                     mark=at position .5 with {\arrow{Stealth[length=2mm]}}},
       dot/.style = {circle, fill, inner sep=2.4pt, node contents={},
                     label=#1},
every edge/.style = {draw, postaction=decorate}
                        ]

\node (s) at (0,0) [dot=$s$];
\node (s') at (2,0) [dot=right:$s'$];

\path      
        (s) edge (s')   
        (s') edge[bend left] (s);
    \end{tikzpicture}
\end{equation}
where both arrows are labeled with the same single mutation, say  $\mu_{j}$, $j\in [1, n]$ and $s'=\mu_{j}(s)$, where $s \in [(X,Q)]$. 
\item All paths in  $\mathbb{G}(Q)$ have a finite number of possible revisits to each vertex. In other words, if $P$ is a path in $\mathbb{G}(Q)$, it is of finite length, meaning every vertex in $P$ will be revisited at most a finite number of times along $P$.
  
 \end{itemize}

\end{defn}

\begin{defn}
 A cluster algebra $\mathcal{A}(Q)$ is called \emph{finite mutation type} if the rooted mutation class $[(X,Q)]$ contains finitely many seeds.
\end{defn}


\section{Rooted Mutations Groups}Fix an initial seed $(X,Q)$.

\begin{defn}
Let $M$ denote the set of all sequences (formal words) formed from the elements of the set ${\mu_{1},\ldots,\mu_{n}}$. A relation on $M$ is a sequence of mutations that preserves every seed in $[(X,Q)]$; such mutation sequences are referred to as \emph{global mutation loops}. The group generated by elements of $M$ subject to global mutation loops as the relations is called the \emph{global mutations group} of $\mathcal{A}(Q)$, and will be denoted by $\mathcal{M}$.
\end{defn}

It is important to note that this definition of global mutation loops diverges from the one provided in [15]. Here, the group relations on $M$ stem from the action on the cluster structure of the entire seed. Conversely, in Definition 3.1 of [15], the relations are determined by the action solely on the cluster structure, $[Q]$, of the quiver $Q$.

\begin{rem} Each element $\mu$ of $M$ corresponds to a unique directed subgraph (path) in the rooted cluster diagraph $\mathbb{G}(Q)$. Such path will be called \emph{the rooted path of $\mu$} and will be denoted by $P_{\mu}$. For simplicity we will call it \emph{the path} of $\mu$. Also, we will be swinging between $\mu$ and $P_{\mu}$ freely.
\end{rem}
\begin{proof} 
                Let $\mu=\mu_{i_{k}}\cdots \mu_{i_{1}}$ be a sequence of mutations. Assign the following path (directed subgraph)  of $\mathbb{G}(Q)$ to $\mu$
\begin{equation}\label{}
   \nonumber P_{\mu}: =\xymatrix{\cdot_{(X,Q)} \ar[r]^{\mu_{i_{1}}} &   \cdot_{\mu_{i_{1}}((X,Q))} \ar[r]^{\mu_{i_{2}}}  & \cdot \cdots \cdot \ar[r]^{\mu_{i_{k}}}& \cdot_{\mu((X,Q))}}.
\end{equation}
The uniqueness follows directly from the well-defined property in the definition of mutations, where each single mutation applied to a specific seed produces a distinct seed.

\end{proof}

\begin{defns}
  Fix an initial seed $(X,Q)$. We have the following 

 \begin{enumerate}
                      
                      \item The set of all seeds that appear on the path $P_{\mu}$ in $\mathbb{G}(Q)$, of a mutations sequence $\mu$, will be called \emph{mutation class} of  $\mu$,  and will be denoted by $[\mu]$, and  $[\mu]_{Q}$ for the set of all quivers only that appear in $[\mu]$.
                      
                      \item If $\mu$ satisfies that $\mu(X,Q)=(X,Q)$, then it will be called a \emph{rooted mutation loop} of $(X,Q)$. The set of all rooted mutation loops of $(X,Q)$ will be denoted by $\mathbf{m}_{Q}$ and the  corresponding set of subgraphs of the digraph $\mathbb{G}(Q)$ will be denoted by $\mathbf{m}_{\widetilde{Q}}$.

  \item \emph{A cancelled cluster variable} in $P_{\mu}$ (respect to $\mu$) is a cluster variable that is  produced by  any sub path  of $P_{\mu}$ of the form  $\xymatrix{\cdot_{(Y, Q')}\ar[r]^{\mu_{j}}&\cdot_{\mu_{j}(Y,Q')} \ar[r]^{\mu_{j}}& \cdot_{(Y,Q')}}$ or a $2$-cycle in $\mathbb{G}(Q)$ (respect to the subsequence $\mu^{2}_{j}$ of $\mu$).   \emph{The cluster set} of a sequence of mutations $\mu$ (respect to $P_{\mu}$), denoted by $\mu_{c}$, is the set of all non-canceled cluster variables produced form $(X, Q)$ over the path of $P_{\mu}$.
      
      \item The following relation $\equiv$ defines an equivalence relation on $\mathbf{m}_{Q}$ (respect to $\mathbf{m}_{\widetilde{Q}}$). Let $\mu$ and $\mu'$ be two elements in $\mathbf{m}_{Q}$ (respect to $P_{\mu}$ and $P_{\mu'}$ in $\mathbf{m}_{\widetilde{Q}}$)  such that $\mu'\neq \overleftarrow{\mu}$. Then  we define

          \begin{equation}\label{}
            \nonumber \mu \equiv \mu' \   \text{if and only if} \   \mu_{c}=\mu'_{c}.
          \end{equation}
  
 In such case, we say that $\mu$ and $\mu'$ are \emph{identical} on $(X,Q)$.   We will denote the digraph $\mathbb{G}(Q)$ subject to the equivalence relation $\equiv$ by  $\mathbb{\overline{G}}(Q)$. 
  \item A  path $P_{\mu}$ of a mutation sequence $\mu=\cdots\mu_{i_{k}}\cdots\mu_{i_{1}},k\geq 1$ (respect to $\mu$) is said to be of \emph{finite cluster order} if there is a mutation $\mu_{i_{m}} \in \{\mu\}$ such that 
  
  \begin{equation}\label{}
   \nonumber \mu_{c}=(\mu_{i_{m}}\cdots\mu_{i_{1}})_{c}.
  \end{equation}
  The smallest such natural number $m$ will be called \emph{the cluster order} of $P_{\mu}$ (respect to $\mu$). In the rest of the article, we will omit the word ``cluster" from the term ``cluster order" if no confusion. 
\end{enumerate}
  \end{defns}

\begin{exams}
\begin{enumerate}
  \item For every simply-laced quiver $Q$, let $i$ and $j$ be any two adjacent vertices in $Q$. Then we have $\mu_{[ij]}\equiv \mu_{[ji]}$, where $\mu_{[ij]}=\mu_{i}\mu_{j}\mu_{i}\mu_{j}\mu_{i}$.
  \item Let $\mu_{i_{k}}\cdots\mu_{i_{1}}$ be a sequence of mutations formed of mutations associated to mutually non-adjacent vertices then $\mu_{i_{k}}\cdots\mu_{i_{1}}\equiv \mu_{i_{\sigma(k)}}\cdots\mu_{i_{\sigma(1)}}$ for any permutation $\sigma$ in the symmetric group $\mathcal{S}_{k}$.
  \item Let $Q$ be the following quiver 
  \begin{equation}\label{}
\nonumber Q=
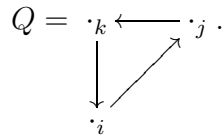
\xymatrix{
\cdot_{k} \ar[d] & \ar[l]\cdot_{j} \\
	\cdot_{i}  \ar[ur]}.
  \end{equation}
Then the mutation sequence $(\mu_{i}\mu_{j}\mu_{k})^{8}$ is  rooted mutation loop for $((x_{i},x_{j},x_{k}),Q)$.
\end{enumerate}
\end{exams}

\begin{prop} If $\mu \equiv \mu'$ then $\{\mu\}=\{\mu'\}$ and  $[\mu]=[\mu']$. In other words, the cluster set characterizes its sequence of mutations. 
\end{prop}
\begin{proof} First we show that $\{\mu\}=\{\mu'\}$
\begin{itemize}
 \item Suppose that there is $\mu_{j} \in \{\mu \} \setminus \{\mu' \}$. Obviously $\mu_{j}$ can not appear as a first mutation in $\mu$. Otherwise the non-cancelable cluster variable $\mu_{j}(x_{j})\in \mu'_{c}$ that is because $\mu_{c}=\mu'_{c}$, which is a contradiction with $\mu_{j} \notin \{\mu'\}$.  Then the  cluster variable produced by $\mu_{j}$ in $\mu_{c}$ is not a first generation cluster variable. Now, assume that $\mu_{j}$ appears after applying some other mutations. Then the cluster variable $x_{j}$ must occur in the denominator vector of some cluster variable in $\mu_{c}$ and then in some denominator vector of the same cluster variable in $\mu'_{c}$. But since $\mu_{j} \notin \{\mu'\}$ then the cluster variable associated to $j$ can not appear in the cluster set of $\mu'$ or in any of the denominator vectors of the cluster variables of $\mu'_{c}$ which contradicts with $\mu_{c}=\mu'_{c}$ thanks to the uniqueness of the denominator vectors. 
\item Secondly we show $[\mu]=[\mu']$. This part is divided into two main parts 
\begin{enumerate}
  \item [(1)] \textbf{The equality of the quivers sets of $\mu$ and $\mu'$. }  Let  $Q^\star \in [\mu]_{Q}$ be a quiver on $P_{\mu}$ that does not belong to the quivers set $[\mu']_{Q}$. In such case, there is a sub sequence of mutations $\mu^{\star}$ of $\mu$ where $\mu^{\star}=\mu_{j}\mu_{i_{k}}\cdots\mu_{i_{1}}$ such that $\mu^{\star}(Q)=Q^\star$. Without loss of generality, we can assume that $\mu_{i_{t}}\cdots\mu_{i_{1}}(Q)$ is a quiver in $[\mu]_{Q}$ for every $t \in [1,k]$, i. e., $Q^\star$  is the first quiver to appears on the path $P_{\mu}$ that is not in  $P_{\mu'}$. One can assume that the change in $Q^\star$, that is different from any quiver in the quiver set of $\mu$, is in the shape of the subquiver $Q'_{Nhb.j}$ that forming the neighborhood of $j$. Let $x^{\star}_{j}$ be the cluster variable produced by applying $\mu^{\star}$. If  $k=0$, then $x^{\star}_{j}$ is a cluster variable of a single mutation sequence. Since  $x^{\star}_{j} \in \mu'_{c}$  and the initial quiver is fixed then $\mu^{\star}(Q)=\mu_{j}(Q)$ is in deed in the quiver set $[\mu']_{Q}$ of $\mu'_{c}$ which is a contradiction. Now, assume that $t\geq1$. Since  $x^{\star}_{j} \in \mu'_{c}$ then we must have a quiver $Q''$ that appears in $[\mu']_{Q}$ over the path $P_{\mu'}$ which has a subquiver that is symmetric or identical to the subquiver formed of the vertices of $Q^\star_{Nhb. j}$. Which means that the difference between $Q''$ and $Q^\star$ appears outside the subquiver of $Q^\star_{Nhb.j}$. Then if all quivers $\mu_{i_{t}}\cdots\mu_{i_{1}}(Q)$ for all $1<t\leq k$ are not similar to any quiver  in $[\mu']_{Q}$, then we would not obtain any subquiver that is similar to $Q^\star_{Nhb.j}$ over the Path $P_{\mu'}$.  Hence, there is $1<t\leq k$ such that $\mu_{i_{t}}\cdots\mu_{i_{1}}(Q)$ is not similar to any quiver in the quiver set $[\mu']_{Q}$ of $\mu^\star$ which contradicts with the assumption that $Q^\star$ is the first such quiver.

  \item [(2)] \textbf{The equality of the sets of  seeds  $[\mu]$ and $[\mu']$. } First we will show that the sets of  clusters of $\mu$ and $\mu'$ are equal considering the clusters as sets. This is an obvious case since if $Y$ is a cluster that does not belong to the cluster set of $\mu'$, then at least one cluster variable in $\mu_{c}$ that is not in $\mu'_{c}$, which is a contradiction.
  Secondly, we show the equality of the clusters sets of $\mu$ and $\mu'$ considering the clusters as $n$-tuples. Suppose that $(Y, Q')$ is not in $[\mu']$ as a whole seed. Since each seed is characterized by its cluster so we will be done if we show that the sets of clusters of $\mu$ and $\mu'$ are identical, which is equivalent to show  that each cluster in $P_{\mu}$ equals a cluster in $P_{\mu'}$ as sets. If $Y$ is a cluster that is not the cluster set of $P_{\mu'}$ then its uniquely associated quiver $Q_{Y}$ is not in the quiver set of $P_{\mu'}$. Hence, $Q_{Y}$ is not in the quiver set of $\mu$ which is a contradiction with the first part (1) of the proof above.
\end{enumerate}

\end{itemize}
\end{proof}

\begin{lem} The following statements hold:
\begin{enumerate}
  \item For any $\mu \in \mathcal{M}$, we have  $(\mu\overleftarrow{\mu})_{c}=1_{c}=\emptyset$. In other words the path  $P_{\mu\overleftarrow{\mu}}$ contains no cluster variables. In particular,  the  sequences of mutations $\mu\overleftarrow{\mu}$ are the only sequences with empty cluster sets. 
\item For any two rooted mutation loops $\mu$ and $\mu'$, one can see that $\mu\mu'$ and $\mu'\mu$ are identical on $(X, Q)$.
  
  \item If $\mathcal{A}(Q)$ is a finite type cluster algebra, then we have the following 
  \begin{enumerate}
    \item any mutation sequence $\mu$ is of the form $\mu=\mu^{k+1}\mu^{(k)}\cdots\mu^{(1)}$ where each of $\mu^{(j)}$ is a rooted mutation loop for $ 0\leq j\leq k$ and $\mu^{k+1}$ is a subsequence of some rooted mutation loop; 
    \item for every cluster variable $y$ there is a rooted mutation loop $\mu$ such that $y \in \mu_{c}$.
  \end{enumerate}
      
\end{enumerate}

\end{lem}

\begin{proof} \begin{enumerate}
                \item The uniqueness is the only part requiring validation. For a non-identity mutation sequence $\mu$, if $\mu_{c}=0$, then every cluster variable produced over $P_{\mu}$ becomes a canceled cluster variable. This implies that $P_{\mu}$ comprises an even-length sequence, with a cyclic process of ``build and break" under the action of $\mu$. Consequently, $\mu$ can be decomposed into two sequences that mirror each other. In this structure, $P_{\mu}$ has two identical seeds in the middle  sandwiched by two other identical seeds and so forth, suggesting $\mu$ is represented as the product $\mu'\overleftarrow{\mu'}$.
   \item The proof of this part follows directly from the definition of rooted mutation loops.  
    \item \begin{enumerate}         
    \item If $\mathcal{A}(Q)$ is a cluster algebra of finite type, then any path $P_{\mu}$ in $\mathbb{G}(Q)$, where $\mu=\mu_{i_{k}}\cdots \mu_{i_{1}}$, encompasses only a finite set of cluster variables. Consequently, $P_{\mu}$ has finite cluster order. Therefore, there exists an index $1\leq m \leq k$ such that $\mu_{c}=(\mu_{i_{m}}\cdots\mu_{i_{1}}){c}$, indicating that $\mu$ ceases to generate new cluster variables after the step $\mu_{i_{m}}$. One would thus expect the initial seed $(X, Q)$ to appear on $P_{\mu}$. If $(X, Q)$ emerges after the step $\mu_{i_{t_{1}}}$, then $\mu=\mu_{i_{k}}\cdots\mu_{i_{t_{1}}+1}\mu^{(1)}$, where $\mu^{(1)}$ denotes the rooted mutation loop $\mu_{i_{t_{1}}}\cdots\mu_{i_{1}}$. Continuing this process of identifying rooted mutation loops in $P_{\mu}$ yields $\mu=\mu^{k+1}\mu^{(k)}\cdots\mu^{(1)}$, where each $\mu^{(j)}$ represents a rooted mutation loop for $0\leq j\leq k$ and $\mu^{k+1}$ denotes the remaining subsequence of mutations, which could be the entire mutation sequence $\mu$ or just one mutation.
\item This observation is a direct consequence of Part (a) above. If $y$ denotes a cluster variable, then a sequence of mutations can be devised that does not cancel out $y$ and eventually reproduces $(X,Q)$.
                 
     \end{enumerate}          
\end{enumerate}
\end{proof}

The following definition is inspired by Proposition 3.5 and Lemma 3.6.
\begin{defns} \begin{itemize}
               \item \textbf{The Reduction Process. } Let $\mathbb{\overline{G}}(Q)$ be the cluster digraph of $(X,Q)$ under the equivalence relation $\equiv$ on $(X,Q)$. We define a sub-digraph $\mathbb{\widetilde{G}}(Q)$ of $\mathbb{\overline{G}}(Q)$ by modifying it through the following process:
               \begin{enumerate} 
                 \item [(1)] Let $P_{\mu}$, be a path in $\mathbb{\overline{G}}(Q)$ where $\mu=\cdots\mu_{i_{k}}\cdots \mu_{i_{1}}$  with a  cluster rank $l< \infty$. Then we identify $\mu$ with $\mu_{i_{l}}\cdots \mu_{i_{1}}$ and consequently identify $P_{\mu}$ with $P_{\mu_{i_{l}}\cdots \mu_{i_{1}}}$;
              
                 \item [(2)]   Let $P_{\mu}$ be a path in $\mathbb{\overline{G}}(Q)$ with $\mu=\mu^{k+1}\mu^{(k)}\cdots\mu^{(1)}$ where each of $\mu^{(j)}$ is a rooted mutation loop of $(X,Q)$. Then  $P_{\mu}$  will be replaced with the reduced path $P_{\mu'}$ where $\mu'$ is obtained from $\mu$ by applying the following steps: 
                     \begin{enumerate}
                       \item First we remove every rooted mutation loop that is identical to $\mu^{(1)}$, and keep $\mu^{(1)}$.
                       \item Next step, is to remove every rooted mutation loop that is identical to $\mu^{(t_{1})}$, and keep $\mu^{(t_{1})}$, where $t_{1}$ is the smallest number such that $\mu^{(t_{1})}$ and $\mu^{(1)}$ are not identical.
                       \item And continue on doing this process until we obtain 
                        
                        \begin{equation}\label{}
                          \nonumber \mu'=\mu^{k+1}\mu^{(t_{l})}\cdots\mu^{(t_{1})}\mu^{(1)},
                        \end{equation}
                        where non of the mutations loops $\mu^{(t_{l})}, \ldots, \mu^{(t_{1})}$, and $\mu^{(1)}$ are  identical on $(X, Q)$. 
                     \end{enumerate}
                   The graph obtained from $\mathbb{\overline{G}}(Q)$ after applying the above reduction process will be denoted by $\mathbb{\widetilde{G}}(Q)$. 
               
              \end{enumerate}

               \item \textbf{The Rooted Mutation Group}. Let $\mathfrak{m}_{Q}$ be the set of all rooted mutation loops of $(X,Q)$ in the reduced cluster diagraph $\mathbb{\widetilde{G}}(Q)$. The group generated by elements of $\mathfrak{m}$  will be called the \emph{rooted mutation group} of $(X,Q)$ and  will  be denoted by $\mathcal{M}(Q)$. 
             \end{itemize}
         
\end{defns}
\begin{exam}
Let $\mu=\mu^{7}\mu^{(6)}\mu^{(5)}\cdots\mu^{(1)}$ where $\mu^{(6)}\equiv \mu^{(4)}$ and $\mu^{(3)}\equiv\mu^{(2)}\equiv\mu^{(1)}$. Then $\widetilde{\mu}=\mu^{7}\mu^{(4)}\mu^{(5)}\mu^{(1)}$.
\end{exam}

    \begin{rems}
               \begin{enumerate}
                 \item The reduction process is automatically applied to the product in the group $\mathcal{M}(Q)$ so the product of any two elements would correspond to a path in $\mathbb{\widetilde{G}}(Q)$.
                 \item The group $\mathcal{M}(Q)$ is a commutative group, thanks to Part 2 of Lemma 3.6. 
               \end{enumerate}
                
             \end{rems}

\begin{prop} For every cluster algebra $\mathcal{A}(Q)$, the rooted mutation group $\mathcal{M}(Q)$ is independent of the choice of the initial seed.  
\end{prop}
\begin{proof} Let $(Y, Q') = \mu((X,Q))$ for some mutation sequence $\mu$, where $\mathcal{M}(Q)$ is the rooted mutation group of the seed $(Y,Q')$. Define $\psi: \mathcal{M}(Q) \rightarrow \mathcal{M}(Q')$ by $\psi(\mu') = \mu \mu' \overleftarrow{\mu}$. The rest of the proof is straightforward.
\end{proof}

\begin{exam}  Let $Q$ be the quiver $\cdot_{1} \longrightarrow \cdot_{2}$. Then $\mathcal{M}(Q)=\{1,\mu_{[12]}\}=\{1,\mu_{[21]}\}$.
               
\end{exam}

\begin{thm} The following are equivalent
\begin{enumerate}
  \item The cluster class $[(X,Q)]$ is finite;
    \item Each path in $\mathbb{\widetilde{G}}(Q)$ (respect to every element in $\mathcal{M}(Q)$)  is of finite  order.
    \item $\mathcal{M}(Q)$ is a finite group and $\mathcal{M}/\mathcal{M}(Q)$ is a finite set;
\end{enumerate}

\end{thm}
\begin{proof}
\begin{enumerate}
  \item $(1) \Rightarrow (2)$.  Let $[(X,Q)]$ be a finite set. Then the number of cluster variables is  finite. Hence, cluster set, $\mu_{c}$, in each path $P_{\mu}$ is also finite. Therefore,  each path $P_{\mu}$ must stop producing new cluster variables at certain point, which means $P_{\mu}$ must have a finite cluster rank.

   \item $(2) \Rightarrow (3)$. Assume that  every path in $\mathbb{\widetilde{G}}(Q)$ is of finite cluster rank. Then, after applying the  second step of the reduction process, all paths in  $\overline{\mathbb{G}}(Q)$ will be of finite length. Hence, each composition of rooted mutation loops must be also finite.   Therefore, the set of elements of $\mathcal{M}(Q)$ is a finite set.
       
      Part 3 (a) of Lemma 3.6, guarantees that, each element $\mu$ of $\mathcal{M}$ can be written of the form $\mu=\mu^{k+1}\mu^{(k)}\cdots\mu^{(1)}$ where $\mu^{(j)}, j\in [1, k]$ are rooted mutation loops for some $k\geq 0$ where $\mu^{k+1}$ is a subsequence of some rooted mutations loop. Step 1 of the reduction process guarantees that  each subsequence of mutations is of finite length, i.e,. in the case of $\mu$, we have $\mu^{k+1}$ is of finite rank which means it will be identified with its longest productive sub sequence.  Then each path $P_{\mu}$ must stop producing new cluster variables at certain point. Therefore the set $\mathcal{M}/\mathcal{M}(Q)$ is also finite.

  \item $(3) \Rightarrow (1)$. First, we will prove that $(3) \Rightarrow (2)$, and then we will show that $(2) \Rightarrow (1)$. Assume that $\mathcal{M}(Q)$ is a finite group and $\mathcal{M}/\mathcal{M}(Q)$ is a finite set.  The reduction process eliminates all redundant subpaths in $\mathbb{\widetilde{G}}(Q)$, preserving only the productive ones, i.e., subpaths containing cluster variables that have not previously appeared in any other subpath. Consequently, the order of each path is proportional to the cardinality of its cluster set. Now, suppose $\mathbb{\widetilde{G}}(Q)$ contains a path of infinite order, denoted by $P_{\mu}$. After applying the reduction process to the corresponding mutation sequence $\mu$, the remaining mutation sequence will still has infinite length. Therefore, $\mu$ is either composed of an infinite number of non-identical rooted mutation loops or is a product of an infinite number of sequences of mutations, which are elements of the quotient $\mathcal{M}/\mathcal{M}(Q)$. This implies that either $\mathcal{M}(Q)$ is an infinite group, or $\mathcal{M}/\mathcal{M}(Q)$ is an infinite set contradicting the assumption that  $\mathcal{M}(Q)$ is a finite group and $\mathcal{M}/\mathcal{M}(Q)$ is a finite set.
      
      Now, suppose that each path (with respect to every element in $\mathcal{M}(Q)$) in $\mathbb{\widetilde{G}}(Q)$ is of finite order. Then, by definition, the cluster set of each path must be finite, and consequently, so are the cluster sets of each sequence of mutations. Furthermore, the number of possible paths in $\mathbb{\widetilde{G}}(Q)$ is finite, as it is an $n$-regular graph with a finite number of edges between any two vertices and no paths of infinite order. Therefore, both the cluster sets of paths and the number of possible paths are finite, leading to a finite set of cluster variables formed by the union of all cluster sets. This implies that the cluster class $[(X,Q)]$ is finite.

\end{enumerate}
\end{proof}

We recall the following concepts and notations before proceeding with the lemma. For a quiver$Q$, $[Q]$ denotes the mutation class of $Q$, consisting of all quivers that can be generated from $Q$ by applying every possible sequence of mutations. The weight of  $[Q]$, denoted by $w[Q]$, is defined as the maximum of the weights of the quivers in $[Q]$. Additionally, $\mathcal{X}_{Q}$ represents the set of all cluster variables of a seed of the form $(X,Q)$. 

\begin{lem} [15, Lemma 3.11] Let $Q$ be a  quiver of  finite mutation type. Then  $w[Q]=4$  if and only if one of the following cases, depending  on the rank of $Q$, is satisfied
   \begin{enumerate}
     \item [(i)] If $rk(Q)=3$  then  $[Q]$ contains  one of the quivers
      \begin{equation}
  \nonumber Q^{x}_{3(1)}=\xymatrix{
\cdot_{t} \ar[d]_{(x, 1)} & \ar[l]_{(1, x)}\cdot_{j} \\
	\cdot_{k}  \ar[ur]_{(2. 2)}}, \ \ \ \ \ \ \ \
Q_{3(2)}= \xymatrix{
\cdot_{t} \ar[d]_{(2, 2)} & \ar[l]_{(2, 2)}\cdot_{j} \\
	\cdot_{k}  \ar[ur]_{(2. 2)}}  \ \ \text{or} \ \ \ \
\ Q_{3(3)}= \xymatrix{
\cdot_{t} \ar[d]_{(2, 1)} & \ar[l]_{(2, 1)}\cdot_{j} \\
	\cdot_{k}  \ar[ur]_{(1. 4)}}
 \end{equation}
 where $x=1, 2, 3$ or $4$.
     \item [(ii)] If $rk(Q)>3$ then $[Q]$ must satisfy the following criteria: Every quiver $Q'\in [Q]$  of weight 4  satisfies the following

       \begin{enumerate}
         \item [A.] Edges of weight  4 in $Q'$  appear in  a cyclic subquiver which is symmetric to  one of the following quivers

 \begin{equation}
\nonumber (a)  \  \   \ Q_{a, x}: \  \xymatrix{
\cdot_{v} \ar[d]_{(x, 1)} & \ar[l]_{(1, x)}\cdot_{j} \\
	\cdot_{k}  \ar[ur]_{(2. 2)}} \  \  \  \ \ \ \  \
 (b)  \  \  \ Q_{a}: \  \xymatrix{
\cdot_{v} \ar[d]_{(1, 2)} & \ar[l]_{(1, 2)}\cdot_{j} \\
	\cdot_{k}  \ar[ur]_{(4. 1)}} \  \  \  \
 \end{equation}
 \begin{equation}\label{}
 (c) \  Q_{c, t}: \  \xymatrix{\cdot_{v}\ar[dr]^{(t, 1)}\\
 \cdot_{k}  \ar[u]^{(1, t)}\ar[dr]_{(1, 2)}& \cdot_{j}\ar[l]_{(2, 2)} \\
 & \ar[u]_{(2, 1)}\cdot_{l} }\\ \  \  \ \  \  \
 (d)  \  Q_{d}:  \  \  \  \ \  \ \xymatrix{\cdot_{v}\ar[dr]\\
 \cdot_{i}  \ar[u]\ar[dr]_{(1, 3)}& \cdot_{j}\ar[l]_{(2, 2)} \\
& \ar[u]_{(3, 1)}\cdot_{l} } \\
 \end{equation}

where $x=1, 2, 3$ or $4$  and  $t=1$ or $ 2$, such that edges of weight 4 are not connected outside their cycles. And any subquiver of the form $ \xymatrix{ \cdot \ar@{-}[r]^{z} &\cdot \ar@{-}[r]^{z} &\cdot}, z=2$ or  $3$  appears in a quiver that is mutationally equivalent to one of the forms in (3.1).
 \item [B.]  $Q'$ will have more than one edge of weight 4  if it  is symmetric to  one of the following cases

      \begin{itemize}
        \item $Q'$ is formed from two or three copies of $Q_{a,1}$  by coherently connecting them at $v$ such as in $X_{6}$ and $X_{7}$.  Or $Q$ is formed from $Q_{a,2}$ and/or $Q_{a}$  in the following form

             \begin{equation}\label{}
              \  \xymatrix{
\cdot_{j'} \ar[r]^{2} &\cdot_{v}\ar[dl]^{2} \ar[dr]_{2} & \ar[l]_{2}\cdot_{j} \\
\cdot_{k'}\ar[u]^{4}&&\cdot_{k}  \ar[u]_{4}}
            \end{equation}

        \item $Q'$ is formed from $Q_{a,1}$, $Q_{a,2}$ and/or $Q_{a}$  in one of the following forms

       \begin{enumerate}
\item [(a)]
            \begin{equation}\label{}
             \  \xymatrix{
\cdot_{j'}\ar[r]^{2}&\cdot_{v'} \ar[d]^{2}\ar@{-}[r] &\cdots\ar@{-}[r]&\cdot_{v} \ar[d]_{2} & \ar[l]_{2}\cdot_{j} \\
&\cdot_{k'}\ar[ul]^{4}&&\cdot_{k}  \ar[ur]_{4}},
            \end{equation}
           where the subquiver connecting $v$ and $v'$ is of $A$-type. Or $Q$ is the quiver
 \item [(b)]

            \begin{equation}\label{}
             \  \xymatrix{
\cdot_{j'}\ar[r]&\cdot_{v'} \ar[d]\cdot\ar@{-}[r]&\cdot_{v} \ar[d] & \ar[l]\cdot_{j} \\
&\cdot_{k'}\ar[ul]^{4}&\cdot_{k}  \ar[ur]_{4}}
            \end{equation}
        \end{enumerate}
      \end{itemize}
 Any additional subquiver of $Q'$ attached to a vertex in the quivers $Q_{a,x}, Q_{a},  Q_{b}$ or $Q_{c,t}$ will be referred to as a `` tail" and we will refer to $Q_{a,x}, Q_{a}, Q_{b}$ or $Q_{c,t}$ as a ``head".

     \item [C.] Tails of a weight-4 quiver $Q'\in [Q]$ satisfy the following
     \begin{itemize}
       \item If $Q'$ is one of the quivers $Q_{a,x}, x=1, 2, 3$, $Q_{a}$ or $Q_{c,1}$ in (3.3)  then $Q'$ could have one simply-laced tail attached at the vertex $v$.

       \item If $Q'$ is one of the quivers in (3.2), (3.3) or (3.4)  then it does not have any tails.
     \end{itemize}

       \end{enumerate}

\end{enumerate}

\end{lem}

\begin{rem} If $\psi: \mathcal{M}(Q)\longrightarrow\mathcal{M}(Q')$ is a group isomorphism, then $\psi$ induces an isomorphism $\phi: \mathcal{M}(Q)\longrightarrow\mathcal{M}(Q')$ that is restricted on $[Q]$. That is if there is a permutation $\sigma$ such that $\mu(Q)=\sigma (\mu'(Q'))$ then $\phi(\mu(Q))=\sigma (\phi(\mu'(Q')))$.
\end{rem}

\begin{thm} Two cluster finite type algebras $\mathcal{A}(X,Q)$ and  $\mathcal{A}(X',Q')$ are isomorphic as cluster algebras if and only if there is a group isomorphism $\psi: \mathcal{M}\rightarrow \mathcal{M}'$ such that the restriction of $\psi$ satisfies the following conditions 
\begin{itemize}
  \item the  rooted mutations groups $\mathcal{M}(Q)$ and $\mathcal{M}(Q')$ are isomorphic;
  \item the two sets $\mathcal{M}/\mathcal{M}(Q)$ and $\mathcal{M'}/\mathcal{M}(Q')$ are in one-to-one correspondence. 
\end{itemize}
  
\end{thm}
\begin{proof}$``\Rightarrow". $ Assume that $\mathcal{A}(X,Q)$ and  $\mathcal{A}(X',Q')$ are isomorphic cluster algebras. Then, there is a seed in the cluster class $[(X',Q')]$ that is symmetric to $(X,Q)$, thanks to  Theorem 3.14 in [14]. Without loss of generality, assume that this seed is $(X',Q')$ itself. Then there is a permutation $\sigma$ such that $(X',Q')=\pm \sigma(X,Q)$. Define $\psi: \mathcal{M}\rightarrow\mathcal{M}$ given by $\mu_{i}\mapsto \mu_{\sigma(i)}, i\in [1, n]$.
\begin{itemize}
  \item First, we will show that $\mathcal{M}(Q)$ and $\mathcal{M}(Q')$ are isomorphic.  Again thanks to  Theorem 3.13 in [14], guarantees that if $\mu$ is a rooted mutation loop for $(X,Q)$ then $\psi(\mu)$ is a rooted mutation loop for $(Y,Q')$. 
      
      Now, we will show that $\psi$ is an isomorphism of groups. Let $\mu^{(1)}=\mu_{i_{k}}\cdots \mu_{i_{1}}, $ and $\mu^{(2)}=\mu_{l_{k'}}\cdots \mu_{l_{1}}, $ where $k, k' \leq n$, and assume that $\mu^{(1)} \equiv \mu^{(2)}$ in $\mathcal{M}(Q)$. Then $\mu^{(1)}(X, Q)=\mu^{(2)}(X, Q)$ and $\mu^{(1)}_{c}=\mu^{(2)}_{c}$. Since $Q'=\pm\sigma(Q)$, hence, $\sigma (\mu^{(1)}(X, Q))=\sigma(\mu^{(2)}(X, Q))$ which means $\mu^{(1)}_{\sigma}(X, Q)=\mu^{(2)}_{\sigma}(X, Q)$, thanks to Lemma 3. 12 in [14]. Now, let $y\in \sigma (\mu^{(1)})_{c}=(\mu^{(1)}_{\sigma})_{c}$. Then there is $t\in [1,k]$ such that $y=\mu_{\sigma (i_{t})}\cdots \mu_{\sigma (i_{1})}(x_{\sigma (j)})$ for some $j\in [1, n]$. Let $x=\mu_{i_{t}}\cdots \mu_{i_{1}}(x_{j})\in \mu^{(1)}_{c}=\mu^{(2)}_{c}$, where $y=\sigma (x)=\mu_{\sigma (i_{t})}\cdots \mu_{\sigma (i_{1})}(x_{\sigma (j)}) \in \mu^{(2)}_{c}$. Hence, there is $z=\mu_{l_{t'}}\cdots\mu_{l_{1}}(x_{j'}) \in \mu^{(2)}_{c}$, where $z=x$. Therefore 
\begin{equation}\label{}
 \begin{array}{ccc}
  \nonumber y &=&\sigma (\mu_{i_{t}}\cdots \mu_{i_{1}}(x_{j}) \\
    &= &\sigma(\mu_{l_{t'}}\cdots\mu_{l_{1}}(x_{j'})\\
   &=&\mu_{\sigma (l_{t'})}\cdots\mu_{\sigma (l_{1})}(\sigma (x_{j'})) \in (\mu^{(2)}_{\sigma})_{c}

\end{array}
\end{equation}

Hence $(\mu^{(1)}_{\sigma})_{c} \subseteq (\mu^{(2)}_{\sigma})_{c}$ and the other direction is similar. Therefore, $(\mu^{(1)}_{\sigma})_{c}=(\mu^{(2)}_{\sigma})_{c}$, and this finishes the proof of $\mu^{(1)}_{\sigma}\equiv \mu^{(2)}_{\sigma}$. 

 \item Secondly, we show that $\mathcal{M}/\mathcal{M}(Q)$ and $\mathcal{M'}/\mathcal{M}(Q')$ are in one-to-one correspondence. Let $\mu$ be a mutation sequence that is not a rooted mutation loop, i. e. , $\mu((X,Q)) \neq (X, Q)$. Then, we have $\sigma \mu((X, Q)) \neq \sigma ((X, Q))$. Hence $\mu_{\sigma} \sigma((X,Q)) \neq \sigma ((X,Q))$. Therefore $\psi (\mu) (Y,Q') \neq (Y, Q')$, and $\psi (\mu)$ is not a rooted mutation loop of $(Y,Q')$. Define the one-to-one correspondence between $\mathcal{M}/\mathcal{M}(Q)$ and $\mathcal{M'}/\mathcal{M}(Q')$ by sending $\mu \mapsto \psi (\mu)$. Since both $\mathcal{X}{Q}$ and $\mathcal{X}{Q'}$ have the same number of cluster variables, Lemma 3.12 in [14] guarantees that $\mu$ and $\psi(\mu)$ have the same cluster order, and the sets $\mu_{c}$ and $\psi(\mu)_{c}$ contain the same number of cluster variables. Therefore, $\psi$ defines a one-to-one correspondence between $\mathcal{M}/\mathcal{M}(Q)$ and $\mathcal{M'}/\mathcal{M}(Q')$, thanks to the fact that $\psi$ is a one-to-one map.
\end{itemize}

$``\Leftarrow". $  Assume that  $\psi: \mathcal{M}\longrightarrow\mathcal{M'}$  is a group isomorphism such that  $\psi: \mathcal{M}(Q)\longrightarrow\mathcal{M}(Q')$ is an isomorphism of rooted mutation groups. Additionally, there is a one-to-one correspondence between $\mathcal{M}/\mathcal{M}(Q)$ and $\mathcal{M'}/\mathcal{M}(Q')$. For simplicity, we will also use $\psi$
 to denote this one-to-one correspondence.  In the following steps we will show that $Q$ and $Q'$ must be symmetric quivers.

\begin{enumerate}

\item \textbf{Case 1. } If $[Q]$ contains a quiver of weight one, this case encompasses the following scenarios: when $[Q]=1$, and if the mutation class $[Q]$ contains any of the quivers $Q=X_{6}, X_{7}, Q=X^{(1,1)}_{6}, X^{(1,1)}_{7}$, $X^{(1,1)}_{8}$ or the quiver in (3.4). Without loss of generality, we will assume that the initial seed is a quiver $Q$ of weight one. In the following, we will show that $\psi$ sends $Q$ to a quiver $Q'$ that is symmetric to $Q$, i. e. , $Q'=\sigma (Q)$ for some permutation $\sigma$.

\begin{itemize}

  \item Step 1: We will show that $\psi$ sends each single mutation in $\mathcal{M}$ to either a single mutation or a sequence of mutation formed of mutual non-adjacent single mutations in $\mathcal{M}$.  
  Let $i\in [1, n]$. We have $(\psi (\mu_{i}))^2=1$. Then $\psi (\mu_{i})=\overleftarrow{\psi (\mu_{i}})$. If $\psi (\mu_{i})$ is one single mutation then nothing to prove. If $\psi (\mu_{i})$ contains two adjacent single mutations, then $(\psi (\mu_{i}))^{2}_{c}\neq 1$ which contradicts with the fact that $\mu^{2}_{i}=1$.
       
  \item Step 2: For every vertex $i \in Q_{0}$, we will show that $\psi$ maps every vertex in $Nhb. {i}$ to a vertex in $\psi(Nhb. {i})$, meaning $\psi$ defines a permutation between the vertices of $Nhb. {i}$ and $\psi(Nhb. {i})$.  
  
  We  start by showing that, if $\mu_{i}$ and $\mu_{j}$ are two adjacent single mutations in $Q$, then $\{\psi(\mu_{i})\}\cap \{\psi(\mu_{j})\} $ contains at least two adjacent vertices in $Q'$. Let $\mu^{i}=\psi(\mu_{i})$ and $\mu^{j}=\psi(\mu_{j})$ and  suppose that $\{\psi(\mu_{i})\}\cap \{\psi(\mu_{j})\} $ contains no adjacent vertices in $Q'$.  Then, we have $\psi (\mu_{[ij]})=\psi(\mu_{i})\psi(\mu_{j})\psi(\mu_{i})\psi(\mu_{j})\psi(\mu_{i})=\psi(\mu_{i})$. But $\{\psi(\mu_{i})\}\cap \{\psi(\mu_{j})\}$ contains no adjacent vertices in $Q'$, hence, we have  $\psi (\mu_{[ij]})=(\psi(\mu_{i}))^{2}(\psi(\mu_{j}))^{2}\mu^{i}=(\psi(\mu_{i})^{2})(\psi(\mu_{j})^{2})\mu^{i}=\mu^{i}$ and, similarly we have $\psi (\mu_{[ji]})=\psi(\mu_{j})=\mu^{j}$, where $\mu^{i}\neq \mu^{j}$ in $\mathcal{M}(Q')$. However,  $\mu_{[ij]}= \mu_{[ji]}$ in  $\mathcal{M}(Q)$. Therefore, for each vertex $i\in Q_{0}$, the subquiver $Nhb. _{i}$ will be transferred to a connected subquiver in $Q'$ plus, possibly, some discrete (disconnected) vertices. But since $Q$ and $Q'$ have the same rank, then in deed $\psi$ creates an isomorphism of graphs from  the underlying  graph of $Q$ to the underlying graph of $Q'$. Therefore, $\psi$ sends each vertex in $Nhb. _{i}$ to exactly one vertex in $\sigma(Nhb. _{i})$, i.e.,  there is a permutation $\sigma$ such that $\psi$ sends each vertex in $Nhb. _{i}$ to a vertex in $\sigma(Nhb. _{i})$.
   \item Step 3: We will  show that $\psi$ preserve the directions in $Q'$, i.e., if $i\longrightarrow j$ in $Q$   then we have $\sigma(i)\longrightarrow \sigma(j)$ in $Q'$.   Which is equivalent   to showing that $Nhb. _{i}$ and $Nhb. _{\sigma(i)}$ are  symmetric for every $i\in Q_{0}$.  Now, if the subquivers of  $Nhb. _{i}$ and $Nhb. _{\sigma(i)}$ are not symmetric, then there is at least one edge in  $Nhb. _{i}$ with opposite direction of its corresponding edge in  $Nhb. _{\sigma(i)}$. Without loss of generality, assume that this edge to be $\cdot_{k} \longleftarrow \cdot_{i} \longrightarrow \cdot_{j}\cdots$ in $Q$ where $\cdot_{\sigma(k)} \longleftarrow \cdot_{\sigma(i)} \longleftarrow \cdot_{\sigma(j)}\cdots$ in $Q'$. One can see that $\psi(\mu_{i}\mu_{k}\mu_{j}) \neq \psi(\mu_{\sigma (i)}\mu_{\sigma (k)}\mu_{\sigma (j)})$, since $\mu_{i}\mu_{k}\mu_{j}$ preserve $Nhb. _{i}$ in $Q$ but $\mu_{\sigma (k)}\mu_{\sigma (j)}$ does not preserve $Nhb. _{\sigma (i)}$ in $Q'$. Therefore, $Q$ and $Q'$ are symmetric which means $\mathcal{A}(Q)$ and $\mathcal{A}(Q')$ are isomorphic. Which finishes the proof for any mutation class  $[Q]$ that contains a simply-laced quiver. 
   
\end{itemize}

\item \textbf{Case 2.} Assume $w([Q])=2. $  If $w([Q])=2$, then we can assume that $Q$ is a quiver with one single edge of weight 2 and the rest of the quiver is a simply-laced subquiver, thanks to Lemma 3.13. Without loss of generality assume that the edge of weight 2 is a leaf $\cdot_{n-1}\longrightarrow \cdot_{n}$.  Then the cluster class $[(X',Q')]$ is also finite, thanks to Theorem 3.12. Hence, $w([Q'])$ must be less than or equals to three. But, since $rk(Q')>2$, therefore $w([Q'])$ can not be three as if $w([Q'])=3$ then $rk(Q')=2$ otherwise $[Q']$ will not be finite. Now, suppose that $w([Q'])=1$. Since all relations in $\mathcal{M}(Q)$ will transfer to $\mathcal{M}(Q')$ and vise versa, then the relation $\mu_{[n-1n]}=\mu_{[nn-1]}$, in $\mathcal{M}(Q')$, transfers to $\mathcal{M}(Q)$ which is impossible as permutation $(n n-1)$ does not belong to $\mathcal{M}(Q)$. Therefore $w([Q'])$ must be two, which means $Q'$ can be obtained from $Q$ by applying some sequence of mutations, i.e., $Q'\in [Q]$, again thanks to Lemma 3.13. Which finishes the proof of this case. 
    
    \item \textbf{Case 3.} 
Assuming that $w([Q])=3$, this case is straightforward because in such a scenario, $[Q]$ contains only one seed of rank 2 and weight 3. Since $\mathcal{A}(Q')$ is of finite type and has the same rank as $Q$ (which is 2) and the same weight of 3, then $(X, Q)$ and $(X',Q')$ must be identical seeds. Therefore, $\mathcal{A}(X, Q)$ and $\mathcal{A}(X',Q')$ are indeed identical cluster algebras.
\item \textbf{Case 4. } Let $w([Q])=4$ and  $\mathcal{A}(X, Q)$ be a finite cluster algebra. One can see both of $[Q]$ and $[Q']$ are of the same weights.  The proof of this case will be divided into several steps based on the rank of $Q$.
    \begin{enumerate}
      \item $rk(Q)=3$. We have several cases as follows 
      \begin{enumerate}
        \item Let $Q=Q^{1}_{3(1)}$. Since $\mathcal{M}(Q)$ contains $\mu_{[tj]}$ and $\mu_{[tk]}$ then $\mathcal{M}(Q')$ must contain them. Hence $Q'$ must have two edges of weight 1. Which means $Q$ and $Q'$ are symmetric. Therefore $\mathcal{A}(X, Q)$ and  $\mathcal{A}(X', Q')$ are isomorphic  cluster algebras. 
            
        \item Let $Q=Q^{2}_{3(1)}$. Since $\mathcal{M}(Q')$ can not be $\mathcal{M}(Q^{1}_{3(1)})$ or $\mathcal{M}(Q_{3(2)})$ because each one of them contains relations that are not in $\mathcal{M}(Q^{1}_{3(1)})$. Therefore there are only two possibilities for $\mathcal{M}(Q')$ which are $\mathcal{M}(Q_{3(3)})$ or $\mathcal{M}(Q^{2}_{3(1)})$. But in deed $Q_{3(3)}$ and $Q^{1}_{3(1)}$ are symmetric and hence $\mathcal{A}(Q_{3(3)})$ and $\mathcal{A}(Q^{1}_{3(1)})$ are isomorphic cluster algebras.
            
        \item Let $Q=Q_{3(2)}$. This is an obvious case as $\mathcal{M}(Q)$ has no relations so non of the other quivers can associate to a mutation group that are isomorphic to it. Recall that $Q'$ is of the same rank and weight of $Q$, hence $Q'$ must be symmetric  or equal to $Q$. Therefore $\mathcal{A}(Q)$ and $\mathcal{A}(Q')$ are isomorphic as cluster algebras. 
            
         \end{enumerate}   
        \item If $rank (Q)>3$.
         
         \begin{enumerate}
           \item The  arguments of the cases when $Q$ contains any of $Q_{a, x}$ or $Q_{a}$  are pretty similar to the cases $Q^{x}_{3(1)}$ and $Q_{3(3)}$. Because in such case $Q$ would be formed of either $Q^{x}_{3(1)}$ or $Q_{3(3)}$ attached to a simply-laced subquiver.
           \item If $Q$ contains $Q_{c, t}, t=1$ or $2$ as a subquiver. Let $t=1$. Since $Q'$ has the same rank and weight as $Q$, and also it must carry the same number and similar positions of simply-laced edges of $Q$. Then $Q'$ must be symmetric to $Q_{c, 1}$ or $Q_{d}$. But $A(Q_{d})$ is actually isomorphic as cluster algebras to $A(Q_{c,1})$ which finishes the proof of this case and the case of if $Q=Q_{d}$ as well.   
               If $t=2$, then in such case the mutation class $[Q]$ is so small and  the proof is very similar to the case of $t=1$. Which means $Q' \in [Q]$ and hence $\mathcal{A}(Q)$ and $\mathcal{A}(Q')$ are isomorphic. 
           \item If $Q$ contains two edges of weight 4. Then $Q'$ must be of the same category,  i.e., has also two edges of weight 4. So, the both of $Q$ and $Q'$ must be symmetric to either $(3.2)$ or $(3.3)$. The case of $Q$ is the quiver in $(3.2)$.  We have $\mathcal{M}$  has the global loops $\mu_{v}, \mu_{j}, \mu_{k}, \mu_{j'}$ and $\mu_{k'}$ which means $\mathcal{M}'$ must have the same global loops. But using Table 1.4 in [16], the only quiver with these global loops is $Q$. Therefore, $Q'\in [Q]$ which finishes the proof of this case. Finally, let $Q$ be symmetric to the quiver in (3.3). Recall that,  the group relations of $\mathcal{M}(Q)$ and $\mathcal{M}(Q')$ are the same, then both of $Q$ and $Q'$ must contain same exact simply-laced subquivers  and same weight two edges. Since $Q'$ has the same rank as $Q$, then the only possible option for $Q'$ is to be identical or symmetric to $Q$, which finishes the proof. 
          
         \end{enumerate}

    \end{enumerate}

\end{enumerate}

\end{proof}
The following is a corollary of the proof of Theorem 3.15.
 \begin{cor} Let $Q$ and $Q'$ be two quivers of finite mutation type. Then,  there is a map $\phi: [Q]\longrightarrow [Q']$ which preserves the rooted/global relations if and only if $Q$ and $Q'$ are symmetric quivers. 
\end{cor}
\begin{que} Do Theorem 3.15 and/or Corollary 3.16 apply to infinite type cluster algebras? .
\end{que}

\subsection*{Acknowledgments}I would like to thank the anonymous reviewer for their thorough review and valuable suggestions, which greatly contributed to the final form of this paper. I also extend my gratitude to Fang Li and Zongzhu Lin for their insightful discussions on the topic.

\end{document}